\documentclass[a4paper,10pt]{article}

\usepackage[english]{babel}
\usepackage[T1]{fontenc}
\usepackage{graphicx}
\usepackage{amsmath, amsfonts,amsthm, mathrsfs, amssymb}
\usepackage{verbatim}
\usepackage{enumerate}
\usepackage[latin1]{inputenc}
\usepackage{hyperref}

\newcommand{\N}{{\mathbb N}}

\newcommand{\R}{{\mathbb R}}

\newcommand{\T}{{\mathbb T}}
\newcommand{\Z}{{\mathbb Z}}

\newcommand{\wtV}{{\widetilde{V}}}

\newcommand{\cN}{{\mathcal N}}
\newcommand{\cO}{{\mathcal O}}
\newcommand{\cR}{{\mathcal R}}
\newcommand{\cT}{{\mathcal T}}
\newcommand{\cV}{{\mathcal V}}

\renewcommand{\d}{\partial}

\newcommand{\la}{\left\langle}
\newcommand{\ra}{\right\rangle}

\def\sleq{\lesssim}
\def\di{{\rm d}}

\newtheorem{theorem}{Theorem}[section]
\newtheorem{lemma}[theorem]{Lemma}
\newtheorem{corollary}[theorem]{Corollary}
\newtheorem{proposition}[theorem]{Proposition}
\newtheorem{definition}[theorem]{Definition}
\newtheorem{remark}[theorem]{Remark}

\title{Almost global existence for the nonlinear Klein-Gordon equation in the nonrelativistic limit} %Title of paper

\author{
 S. Pasquali \newline
\footnote{ \textit{Email: } \texttt{stefano.pasquali@unimi.it} }
}

\begin{document}

\maketitle

\begin{abstract}
We study the one-dimensional nonlinear Klein-Gordon (NLKG) equation 
with a convolution potential, and we prove that solutions 
with small $H^s$ norm remain small for long times. 
The result is uniform with respect to $c \geq 1$, 
which however has to belong to a set of large measure. 
\end{abstract}

\maketitle %\maketitle must follow title, authors, abstract and \pacs

\section{Introduction} \label{intro}

In this paper we study the nonlinear Klein-Gordon (NLKG) equation 
in the nonrelativistic limit, namely as the speed of light $c$ 
tends to infinity. 
Formal computations going back to the the first half of the
last century suggest that, up to corrections of order $\cO(c^{-2})$, the
system should be described by the nonlinear Schr\"odinger (NLS) 
equation. Subsequent mathematical results have shown that the NLS 
describes the dynamics over time scales of order $\cO(1)$. \\

The nonrelativistic limit for the Klein-Gordon equation 
on $\R^d$ has been extensively studied over more then 30 years, and
essentially all results only show convergence of the solutions of NLKG 
to the solutions of the approximate equation fortimes of order $\cO(1)$. 
The typical results, due to Masmoudi, Machihara, Nakanishi and Ozawa, 
ensure convergence locally uniformly in time, either with loss of regularity  
(see \cite{tsutsumi1984nonrelativistic}, \cite{najman1990nonrelativistic} and \cite{machihara2001nonrelativistic}) 
or without loss of regularity (see \cite{machihara2002nonrelativistic}, \cite{masmoudi2002nonlinear}).

We also mention the recent papers \cite{lu2016partially} and \cite{pasquali2017dynamics}, 
which discuss the long-time convergence of solution of NLKG in the 
nonrelativistic limit on $\R^d$: however, the results proved in both papers 
have some limitations, either on the nonlinearity (in \cite{lu2016partially} 
the authors studied only the quadratic NLKG) or on the particular form of the 
solution (in \cite{pasquali2017dynamics}).

\vskip 10pt 

Concerning the nonrelativistic limit of the NLKG on the $d$-dimensional torus 
$\T^d$, we mention the work by Faou-Schratz \cite{faou2014asymptotic}, in which 
the authors were able to justify the approximation of the solutions of NLKG 
by solutions of the NLS over time scales of order $\cO(1)$ through a variant 
of Birkhoff Normal Form theory. 
In \cite{pasquali2017dynamics} Birkhoff Normal Form Theory was exploited 
in order to recover the result by Faou-Schratz, and to generalize 
it to all smooth compact Riemannian manifolds.

Actually, \cite{faou2014asymptotic} dealt with the construction of 
numerical schemes which are robust in the nonrelativistic limit; 
we refer also to \cite{bao2012analysis}, \cite{bao2016uniformly} and to 
\cite{baumstark2016uniformly} for the numerical analysis of the 
nonrelativistic limit of the NLKG equation. \\

In this paper we generalize the techniques developed in 
\cite{bambusi2006birkhoff} 
in order to prove a long-time existence result for the NLKG with a 
convolution potential with Dirichlet boundary conditions, 
uniformly in $c \geq 1$.

An immediate corollary of our result allows us to show that 
for any $\alpha>0$ any solution in $H^s$ with initial datum of size 
$\cO(c^{-\alpha})$ remains of size $\cO(c^{-\alpha})$ up to times of 
order $O(c^{\alpha(r+1/2)})$ for any $r \geq 1$; 
however, we have to assume that both the parameter $c$ 
and the coefficients of the potential belong to a set of large measure. 
The main limitation of such a result is that it holds only for solutions 
with initial data which are small with respect to $c$. 

The new ingredient in the proof with respect to \cite{bambusi2006birkhoff} 
is a diophantine type estimate for the frequencies, 
which holds uniformly when $c\to\infty$. 

An aspect that would deserve future work is the study of the
nonrelativistic limit of the NLKG without potential. This is expected
to be a quite subtle problem since, for $c \neq 0$ the frequencies of
NLKG are typically non resonant, while the limiting frequencies are
resonant. The issue of long-time existence for small solutions of the NLKG 
equation without potential on compact manifolds has received a lot of interest; 
see for example \cite{delort2004long}, \cite{delort2009long}, \cite{fang2010long}, 
\cite{fang2017almost} and \cite{delort2017long}. 
However, all results in the aforementioned papers rely on a nonresonance 
condition which is not uniform with respect to $c$. \\

The paper is organized as follows.
In sect. \ref{results} we state the results of the paper, together with 
some examples and comments. In sect. \ref{proofBNFtorus} we prove our result 
for the NLKG with a convolution potential on $I:=[0,\pi]$. \\

\emph{Acknowledgements.} This work is based on author's PhD thesis. 
He would like to express his thanks to his supervisor Professor Dario Bambusi.

\section{Statement of the Main Results} \label{results}

Now consider the following equation: 
\begin{align} \label{NLKGpot}
\frac{1}{c^2} \; u_{tt} \; - \; u_{xx} \; + \; c^2 \; u \; + V \ast u + \d_u f(u) &= 0, 
\end{align}
with $c \in [1,+\infty)$, $x \in I$, $f \in C^\infty(\R)$ a real-valued 
function with a zero of order four at the origin, 
with Dirichlet boundary conditions. 
The potential has the form 
\begin{align} \label{coeffpot}
V(x) &= \sum_{j \geq 1} \; v_j \; cos(jx).
\end{align}
\indent By using the same approach of \cite{bambusi2006birkhoff}, 
we fix a positive $s$, and for any $M>0$ we consider the probability space
\begin{align} \label{probspace}
\cV := \cV_{s,M} &= \left\{ (v_j)_{j\geq1} \; : \; v'_j \; := \; M^{-1}j^{s}v_j \in \left[-\frac{1}{2},\frac{1}{2}\right] \right\},
\end{align}
and we endow the product probability measure on the space of $(c,(v'_j)_j)$.

We recall that in this case the frequencies are given by 
\begin{align} \label{freq}
\omega_j := \omega_j(c) &= c \sqrt{c^2+\lambda_j} \; = \; c^2 \; + \frac{\lambda_j}{1+\sqrt{1+\lambda_j/c^2} } \\
&= \; c^2 \; + \frac{\lambda_j}{2} -\frac{\lambda_j^2}{2c^2} \frac{1}{(1+\sqrt{1+\lambda_j/c^2})^2 },
\end{align}
where $\lambda_j = j^2 + v_j$. 
Now we introduce the following change of coordinates, 
\begin{align}
 \psi &:= \frac{1}{\sqrt{2}} \left[ \left(\frac{(c^2-\Delta+\wtV)^{1/2}}{c} \right)^{1/2} u - i \left(\frac{c}{(c^2-\Delta+\wtV)^{1/2}}\right)^{1/2}u_t \right], \label{changevartorus}
\end{align}
where $\wtV$ is the operator that maps $u$ to $V \ast u$.
The Hamiltonian of \eqref{NLKGpot} now reads
\begin{align} \label{NLKGhamtorusnew}
H(\psi,\bar\psi) &= \la \bar{\psi}, c(c^2-\Delta+\wtV)^{1/2}\psi \ra 
+ \int_I f \left( 
\left( \frac{c}{(c^2-\Delta+ \wtV)^{1/2}} \right)^{1/2} \frac{\psi+\bar\psi}{\sqrt{2}} 
\right) \di x.
\end{align}
Therefore the Hamiltonian takes the form 
\begin{align} \label{NLKGtorus}
H(\psi,\bar\psi) &= H_0(\psi,\bar\psi) + N(\psi,\bar\psi),
\end{align}
where

\begin{align}
H_0(\psi,\bar\psi) &=  \la \bar\psi, c (c^2-\Delta+\wtV)^{1/2}\psi \ra, \label{H_0} \\
N(\psi,\bar\psi) &= \int_{I} f \left( \left(\frac{c}{ (c^2-\Delta+ \wtV)^{1/2} }\right)^{1/2} (\psi+\bar\psi) \right) \di x, \label{Nlin} \\
&\sim \sum_{l \geq 4} \int_{I }N_l(x) \left( \left(\frac{c}{ (c^2-\Delta+\wtV)^{1/2} }\right)^{1/2} (\psi+\bar\psi) \right)^l \di x,
\end{align}
where $N_l \in C^\infty$ for each $l$ (since $V \in C^\infty$), and 
\begin{align*}
\left(\frac{c}{ (c^2-\Delta+\wtV)^{1/2} }\right)^{1/2}: H^s \to H^s
\end{align*}
is a smoothing pseudodifferential operator, which can be estimated uniformly in 
$c \geq 1$. 

\begin{theorem} \label{BNFtorusthm}
Consider the equation \eqref{NLKGpot} and fix $\gamma>0$, and $\tau>1$. 
Then for any $r \geq 1$ there exists $s^\ast >0$ and, for any $s>s^\ast$, 
there exists a set $\cR_\gamma:=\cR_{\gamma,s,r} \subset \; [1,+\infty) \times \cV$ 
satisfying 
\begin{align*}
|\cR_{\gamma} \cap ([n,n+1] \times \cV)| &= \cO(\gamma) \; \; \forall n \in \N_0,
\end{align*}
and there exists $R_s>0$ such that for any 
$(c,(v_j)_j) \in ([1,+\infty) \times \cV)\setminus \cR_\gamma$ and 
for any $R<R_s$ there exist $N:=N(r,R)>0$, and a canonical transformation 
\begin{align*}
\cT_c:=\cT^{(r)}_{c} &: B_s(R/3) \to B_s(R)
\end{align*}
 such that 
\begin{align*}
H_r := H \circ \cT_c^{(r)} &= H_0 + Z^{(r)} + R^{(r)},
\end{align*}
where $Z^{(r)}$ is a polynomial of degree (at most) $r+2$, such that
\begin{align}
Z^{(r)}(\psi,\bar\psi) &= \sum_{n \in \N^\N} Z_{n}\psi^n \bar\psi^n, \nonumber \\
Z_{n} \neq 0 &\Longrightarrow \sum_{l \geq N+1} n_l \leq 2, \label{NFtorusprop}
\end{align}
and such that
\begin{align} 
\sup_{ B_s( R/3 ) } \|X_{R^{(r)}}(\psi,\bar\psi)\|_{H^s} &\leq K_s \; R^{r+3/2}, \label{remtorusest} \\
\sup_{ B_s( R/3 ) } \|\cT^{(r)}_c-id\|_{H^s} &\leq K_s \; R^2. \label{CTtorusest}
\end{align}
and we have that $Z^{(r)}$ depends on the actions $I = \psi \, \bar\psi$ only. 
Moreover, there exist $K^\ast_s>0$ and $K'>0$ such that if the initial datum satisfies 
\begin{align} \label{indatumass}
\|(\psi_0,\bar\psi_0)\|_{H^s} &\leq K \; R
\end{align}
with $K < K^\ast_s$, then 
\begin{align}
\|(\psi(t),\bar\psi(t))\|_{H^s} &\leq 2K \; R, \; \; |t| \leq K' \, R^{-(r+1/2)} \label{smallsoltorus} \\
\|(I(t),\bar I(t))\|_{H^s} &\leq K \; R^3, \; \; |t| \leq K' \, R^{-(r+1/2)}. \label{acttorus}
\end{align}
Finally, there exists a smooth torus $I_c$ such that for any $s_1 < s-1/2$ 
there exists $K_{s_1}>0$ such that
\begin{align}
d_{s_1}( (\psi(t),\bar\psi(t)), I_c )  &\leq K_{s_1} \, R^{\frac{r_1}{2} + 1}, \; \; |t| \leq K' \, R^{-(r-r_1+1/2)},
\end{align}
where $r_1 \leq r$, and $d_{s_1}$ is the distance in $H^{s_1}$.
\end{theorem}

\begin{remark}
The fact that $Z$ depends only on the actions is a direct consequence of the 
non-resonance property established in Theorem \ref{cond2mel}.
\end{remark}

\begin{remark}
It would also be interesting to study the dependence of $I_c$ on $c$. 
One could expect that it should converge to an invariant torus of the NLS with 
a convolution potential. We expect this fact to be true, but it needs further 
investigation for a proof. This is due to the fact that the NLS is the 
{\it singular} limit of NLKG and to the fact that $c$ is only allowed to vary 
in Cantor like sets, so that one can only expect a Whitney-smooth dependence.
\end{remark}

By exploiting the same argument used to prove Theorem \ref{BNFtorusthm} 
one can immediately deduce the following almost global existence result for 
solutions with small (with respect to $c$) initial data.

\begin{corollary}
Fix $\alpha>0$, $\gamma>0$, and $\tau>1$. 
Then for any $r \geq 1$ there exists $s^\ast >0$ and, for any $s>s^\ast$, 
there exists a set $\cR_{\gamma,\tau,s,\alpha,r} \subset \; [1,+\infty) \times \cV$ 
 such that there exists $c^\ast>0$ such that for any 
$(c,(v_j)_j) \in ([1,+\infty) \times \cV)\setminus \cR_{\gamma,\tau,s,\alpha,r}$ 
with $c>c^\ast$ the following holds: there exist $K^\ast_s>0$ and $K'>0$ 
such that if the initial datum satisfies 
\begin{align*}
\|(\psi_0,\bar\psi_0)\|_{H^s} &\leq \frac{K}{c^\alpha}
\end{align*}
with $K < K^\ast_s$, then 
\begin{align*}
\|(\psi(t),\bar\psi(t))\|_{H^s} &\leq \frac{2K}{c^\alpha}, \; \; |t| \leq K' \, c^{\alpha(r+1/2)}, \\
\|(I(t),\bar I(t))\|_{H^s} &\leq \frac{K}{c^{3\alpha}}, \; \; |t| \leq K' \, c^{\alpha(r+1/2)}.
\end{align*}
\end{corollary}

\section{Proof of Theorem \ref{BNFtorusthm} } \label{proofBNFtorus}
\sectionmark{Proof of Theorem}

In order to prove Theorem \ref{BNFtorusthm}, we need to show some 
nonresonance properties of the frequencies $\omega=(\omega_j)_{j>0}$: it will be 
crucial that these properties hold uniformly (or at least, up to a set of small 
probability) in $(1,+\infty) \times \cV$, since this will allow us to deduce 
a result which is valid in the nonrelativistic limit regime. 

We use the notation ``$a \sleq b$'' (resp. ``$a \gtrsim b$'') to mean 
``there exists a constant $K>0$ independent of $c$ such that $a \leq Kb$'' 
(resp. $a \geq Kb$).

\begin{proposition} \label{cond0pot}
Let $r\geq1$, $c\geq1$ be fixed. 
Then $\forall \gamma>0$ $\exists \cV'_{s,M,\gamma} \subset \cV$ 
with $|\cV \setminus \cV'_{s,M,\gamma}|=\cO(\gamma)$, and 
$\exists \tau>1$ s.t. $\forall (v_j)_{j \geq 1} \in \cV'_{s,M,\gamma}$ and 
$\forall$ $N\geq1$ 
\begin{align} \label{nonres0prop}
|\omega \cdot k \; + \; n| &\geq \frac{\gamma}{N^\tau}
\end{align}
 for $0<|k|\leq r$, $supp(k) \subseteq \{1,\ldots,N\}$, 
and $\forall$ $n \in \mathbb{Z}$. \\
\end{proposition}

\begin{proof}
Let $p_{k,n}((v_j)_{j\geq1}) \; := \; \sum_{j=1}^N \omega_jk_j + n$, 
and assume that $k_h \neq 0$ for some $h$. Then 
\begin{align*}
\left| \frac{\d p_{k,n}}{\d v'_h} \right| = \left| \frac{k_h h^{-s}}{2 \sqrt{1+\lambda_h/c^2}} \right| &\gtrsim \; \frac{1}{2 h^s \sqrt{1+h^{max(s,2)}} } \geq \frac{1}{2 N^s \sqrt{1+N^{max(s,2)}}} \; > \; 0, 
\end{align*}
hence by Lemma 17.2 of \cite{russmann2001invariant}
\begin{align}
\left| \{(v'_j)_{j\geq1} \; : \; |p_{k,n}((v_j)_{j\geq1})| < \gamma_0\} \right| &\leq  \gamma_0 \; N^{s+max(s,2)/2}; \nonumber \\
\left| \bigcup_{|n| \leq rN} \; \bigcup_{\stackrel{0<|k|\leq r}{supp(k)\subseteq\{1,\ldots,N\}}} \; \{(v'_j)_{j\geq1} \; : \; |p_{k,n}((v_j)_{j\geq1})| < \gamma_0\} \right| &\leq \gamma_0 \; N^{r+1+s+max(s,2)/2} \label{estresset}, 
\end{align}
and by choosing $\gamma_0 \; = \; \frac{\gamma}{N^\tau}$ with 
$\tau > s+r+2+max(s,2)/2$ we get that the left hand side of \eqref{estresset} 
is bounded, and therefore the thesis holds.
\end{proof}

\begin{proposition} \label{cond1mel}
Let $r\geq1$ be fixed. Then $\forall \gamma>0$ there exists a set 
$\cR_\gamma:=\cR_{\gamma,s,r} \subset \; [1,+\infty) \times \cV$ satisfying 
\begin{align*}
|\cR_{\gamma} \cap ([n,n+1] \times \cV)| &= \cO(\gamma) \; \; \forall n \in \N_0,
\end{align*}
and $\exists \tau>1$ such that 
$\forall (c,(v_j)_j) \in ([1,+\infty) \times \cV) \setminus \cR_{\gamma}$ 
and $\forall$ $N \geq 1$
\begin{align} \label{nonres1prop}
\left|\sum_{j=1}^N \; \omega_j k_j \; + \; \sigma \omega_l\right| &\geq \frac{\gamma}{N^\tau}
\end{align}
 for $0<|k|\leq r$, $supp(k) \subseteq \{1,\ldots,N\}$, $\sigma=\pm 1$, 
$l \geq N$.
\end{proposition}

\begin{proof}
Without loss of generality, we can choose $\sigma=-1$.

Now fix $k \in \Z^N$ with $0<|k|\leq r$, and fix $l \geq N$. 
Set $p_{k,l}(c,(v_j)_{j\geq1}) \; := \; \sum_{j=1}^N \omega_jk_j-\omega_l$. 
We can rewrite the function $p_{k,l}$ in the following way: 
\begin{align*}
p_{k,l}(c,(v_j)_{j\geq1}) \; &= \alpha c^2 \; + \; \sum_{j=1}^N \frac{k_j\lambda_j}{1+\sqrt{1+\lambda_j/c^2}} \; - \frac{\lambda_l}{1+\sqrt{1+\lambda_l/c^2}},
\end{align*}
where $\alpha:= (\sum_{j=1}^Nk_j)-1 \in \{-r-1,\ldots,r-1\}$. 
Now we have to distinguish some cases: 

\emph{Case $\alpha=0$}: in this case we have that $p_{k,l}$ can be small 
only if $l^2 \leq 3(N^2+N^s)^2r^2$. So to obtain the result we just apply 
Proposition \ref{cond0pot} with $N':=\sqrt{3}(N^2+N^s)r$, $r'=r+1$. 

\emph{Case $\alpha\neq0$, $c \leq \lambda_N^{1/2}r^{1/2}$}: we have that 
\begin{align*}
\sum_{j=1}^N \; c\sqrt{c^2+\lambda_j}k_j  &\leq  r \sqrt{c^4+c^2\lambda_N} \; \leq \; \sqrt{2} \; r^2 \lambda_N,
\end{align*}
so $|\sum_{j=1}^N\omega_jk_j-\omega_l|$ can be small only for $l^2<rN^2$. 
Therefore, in order to get the thesis we apply Proposition \ref{cond0pot} 
with $N':=\sqrt{r}N$, $r':=r+1$. 

\emph{Case $\alpha > 0$, $c > \lambda_N^{1/2}r^{1/2}$}: first notice that 
if we set $f(x):=\frac{x^2}{2\sqrt{1+x}(1+\sqrt{1+x})^2 }$, and we put 
$x_j:=\lambda_j/c^2$, in this regime we get 
\begin{align*}
\left|\sum_{j=1}^Nk_jf(x_j)\right| &\leq \frac{r}{2} f(x_N) \leq \frac{1}{2}.
\end{align*}
Now define $\tilde{p}_{k,l}(c^2) := \alpha c^2 \; - \; \frac{\lambda_l}{1+\sqrt{1+\lambda_l/c^2}}$. One can verify that 
\begin{align*}
\tilde{p}_{k,l}(c^2) &=  0; \\
c^2=c^2_{l,\alpha} &:= \frac{\lambda_l}{\alpha(\alpha+2)},
\end{align*}
and that
\begin{align*}
\frac{\d \tilde{p}_{k,l}}{\d (c^2)}(c^2_{l,\alpha}) &= 
 \alpha - \frac{\alpha^2(\alpha+2)^2}{2\sqrt{1+\alpha(\alpha+2)}(1+\sqrt{1+\alpha(\alpha+2)})^2}  > 0.
\end{align*}
Besides, in an interval $\left[ c^2_{l,\alpha}-\frac{\varrho}{\alpha(\alpha+2)},c^2_{l,\alpha}+\frac{\varrho}{\alpha(\alpha+2)} \right] \; =: \; [c^2_{l,\alpha,-},c^2_{l,\alpha,+}]$ we have that 
\begin{align*}
\frac{\d \tilde{p}_{k,l}}{\d (c^2)}(c^2) &> \left(\frac{1}{2}+\frac{1}{2(r+1)}\right) \alpha. 
\end{align*}
Then, by exploiting Lemma 17.2 of \cite{russmann2001invariant}, we get that  
\begin{align*}
\left| \left\{ c^2 \in B\left(c^2_{l,\alpha},\frac{\varrho}{\alpha(\alpha+2)}\right] : |\tilde{p}_{k,l}(c^2)| \leq \gamma \right\}\right| &\leq \gamma \frac{2(r+1)}{(r+2)\alpha}
\end{align*}
for any $\gamma>0$ s.t. $\gamma \frac{2(r+1)}{(r+2)\alpha} < \frac{\varrho}{\alpha(\alpha+2)}$; $\gamma < \frac{(r+2)\varrho}{2(r+1)(\alpha+2)}$. 

Now, since in this regime 
$\left|\frac{\d (p_{k,l}-\tilde{p}_{k,l})}{\d c^2}\right| \leq \frac{1}{2}$, 
we can deal with $p_{k,l}$ in a similar way as before, and we can conclude that
\begin{align}
\lim_{\gamma\to0}\left| \bigcup_{\stackrel{0<|k|\leq r}{\text{supp}(k)\subseteq\{1,\ldots,N}} \; \bigcup_{l\geq N} \; \{c^2 \in [c^2_{l,\alpha,-},c^2_{l,\alpha,+}] \; : |p_{k,l}(c^2)| \leq\gamma\} \right| &= 0.
\end{align}

\emph{Case $\alpha<0$, $c > \lambda_N^{1/2}r^{1/2}$}: since
\begin{align*}
\left| \sum_{j=1}^N \frac{k_j\lambda_j}{1+\sqrt{1+\lambda_j/c^2}} \right| &\leq \frac{r\lambda_N}{2} \leq \frac{c^2}{2},
\end{align*}
we have that $p_{k,l}$ can be small only if $\lambda_N < r \lambda_N$. 
So, in order to get the result, we apply Proposition \ref{cond0pot} 
with $N':=r^{1/2}N$, $r':=r+1$. 
\end{proof}

\begin{theorem} \label{cond2mel}
Let $r\geq1$ be fixed. Then $\forall \gamma>0$ there exists a set 
$\cR_\gamma:=\cR_{\gamma,s,r} \subset \; [1,+\infty) \times \cV$ satisfying 
\begin{align*}
|\cR_{\gamma} \cap ([n,n+1] \times \cV)| &= \cO(\gamma) \; \; \forall n \in \N_0,
\end{align*}
and $\exists \tau>1$ such that 
$\forall (c,(v_j)_j) \in ([1,+\infty) \times \cV) \setminus \cR_{\gamma}$ 
and $\forall$ $N \geq 1$
\begin{align} \label{nonres2prop}
\left|\sum_{j=1}^N \; \omega_j k_j \; + \; \sigma_1 \omega_l \; + \; \sigma_2 \omega_m \right| &\geq \frac{\gamma}{N^\tau}
\end{align}
 for $0<|k|\leq r$, $supp(k) \subseteq \{1,\ldots,N\}$, 
$\sigma_1,\sigma_2 \in \{\pm 1\}$, $m > l \geq N$.
\end{theorem}

\begin{proof}
If $\sigma_i=0$ for $i=1,2$, then we can conclude by using Proposition 
\ref{cond1mel}. \\
\indent Now, consider the case $\sigma_1=-1$, $\sigma_2=1$, and denote 
\begin{align*}
p_{k,l,m}(c^2)  &:= \sum_{j=1}^N \; \omega_j(c^2) k_j \; - \omega_l(c^2) \; +  \omega_m(c^2).
\end{align*}
Now fix $\delta>3$. If $m\lesssim N^\delta$, then we can conclude by applying 
Proposition \ref{cond0pot} and \ref{cond1mel}. 
So from now on we will assume that $m,l>N^\delta$. 

We have to distinguish several cases: 

\emph{Case $c < \lambda_l^\alpha$:} we point out that, since 
\begin{align*}
c \sqrt{c^2+\lambda_l} = c \lambda_l^{1/2} \sqrt{1+\frac{c^2}{\lambda_l}} &=  c \lambda_l^{1/2} \left( 1+ \frac{c^2}{2\lambda_l} + O\left(\frac{1}{\lambda_l^2}\right) \right),
\end{align*}
we get (denote $m=l+j$)
\begin{align*}
\omega_m-\omega_l &=  j c \; + \; \frac{1}{2}\left(\frac{v_m}{m}-\frac{v_l}{l}\right) \; + \; \frac{c^3}{2\lambda_l^{1/2}} - \frac{c^3}{2\lambda_m^{1/2}} + O\left(\frac{1}{m^3}\right) + O\left(\frac{1}{l^3}\right),
\end{align*}
that is, the integer multiples of c are accumulation points for the differences 
between the frequencies as $l,m \to \infty$, provided that $\alpha<\frac{1}{6}$.

\emph{Case $c > \lambda_m$:} in this case we have 
(again by denoting $m=l+j$) that 
$\lambda_m-\lambda_l = j(j+2l)+(v_m-v_l) \; = \; 2jl+j^2+a_{lm}$, with $|a_{lm}| \leq \frac{C}{l}$, so that
\begin{align*}
p_{k,l,m} &= \sum_{h=1}^N \omega_hk_h \; \pm 2jl  \pm j^2 \pm a_{lm}.
\end{align*}
If $l > 2 \,C N^{\tau}/\gamma$ then the term $a_{lm}$ represents a negligible 
correction and therefore we can conclude by applying Proposition \ref{cond0pot}.
 On the other hand, if $l \leq 2 \, C N^{\tau}/\gamma$, we can apply the same 
Proposition with $N':= 2 C N^{\tau}/\gamma$ and $r':=r+2$. 

\emph{Case $\lambda_l^{1/6} \leq c \lesssim \lambda_l^{1/2}$:} 
if we rewrite the quantity to estimate 
\begin{align*}
p_{k,l,m}(c^2) &= \alpha c^2 \; + \; \sum_{h=1}^N \frac{\lambda_hk_h}{1+\sqrt{1+\frac{\lambda_j}{c^2} }} \; + \; \omega_m \; - \; \omega_l,
\end{align*}
where $\alpha:=\sum_{h=1}^Nk_h$, we distinguish three cases:
\begin{itemize}
\item if $\alpha>0$, then we notice that
\begin{align*}
\left| \sum_{h=1}^N \frac{\lambda_hk_h}{1+\sqrt{1+\frac{\lambda_j}{c^2} }} \right| \leq \frac{r \lambda_N}{1+\sqrt{1+\lambda_N/c^2}} &\leq \frac{r\lambda_N}{1+\sqrt{1+\lambda_N/\lambda_l}} \leq \frac{r\lambda_N}{2},
\end{align*}
\begin{align*}
|\omega_m-\omega_l| &= c \frac{\lambda_m-\lambda_l}{\sqrt{c^2+\lambda_m}+\sqrt{c^2+\lambda_l}} \stackrel{m>l}{\geq} \frac{c\lambda_l^{1/2}}{\sqrt{c^2+\lambda_m}+\sqrt{c^2+\lambda_l}} \\
&\gtrsim \frac{N^{\delta/3} \lambda^{1/2}}{\sqrt{N^{2\delta/3}+\lambda_m^{1/2}}+\sqrt{N^{2\delta/3}+\lambda_l^{1/2}}} >0,
\end{align*}
thus $|p_{k,l,m}|>|\lambda_l^{1/3}-\frac{r}{2}\lambda_N|>0$, since $l>N^3$; \\
\item if $\alpha=0$, then we just notice that 
\begin{align*}
|\omega_m-\omega_l| \geq \gamma(\lambda_m-\lambda_l) &\stackrel{m>l}{\gtrsim} \; \gamma_0 \; \lambda_l^{1/2},
\end{align*}
which is greater than $\gamma_0/N^\tau$ for $\tau>-1$, since $l>N^3$; \\
\item if $\alpha<0$, then we just recall that 
$|\omega_m-\omega_l|>\gamma_0\lambda_l^{1/2}$, and by choosing $\gamma_0$ 
sufficiently small (actually $\gamma_0\leq|\alpha|$) we get that 
also in this case $p_{k,l,m}$ is bounded away from zero.
\end{itemize}

\end{proof}

The proof is based on the method of Lie transform. Let $s > s^\ast$ be fixed. 

Given an auxiliary function $\chi$ analytic on $H^{s}$, 
we consider the auxiliary differential equation
\begin{align} \label{auxDEs}
\dot \psi &= i\nabla_{\bar\psi} \chi(\psi,\bar\psi) =: X_\chi(\psi,\bar\psi)
\end{align}
and denote by $\Phi^t_\chi$ its time-$t$ flow. 
A simple application of Cauchy inequality gives

\begin{lemma} \label{cauchy}
Let $\chi$ and its symplectic gradient be analytic in $B_{s}(\rho)$. 
Fix $\delta<\rho$, and assume that 
\begin{align*}
\sup_{B_{s}(R)} \|X_\chi(\psi,\bar\psi)\|_{s} &\leq \delta.
\end{align*}
Then, if we consider the time-$t$ flow $\Phi^t_\chi$ of $X_\chi$ we have that 
for $|t| \leq 1$ 
\begin{align*}
\sup_{B_{s}(R-\delta)} \|\Phi^t_\chi(\psi,\bar\psi)-(\psi,\bar\psi)\|_{s} &\leq \sup_{B_{s}(R)} \|X_\chi(\psi,\bar\psi)\|_{s}.
\end{align*}
\end{lemma}

The map $\Phi := \Phi^1_\chi$ will be called the \emph{Lie transform} 
generated by $\chi$.

Given a homogeneous polynomial $f$ of degree $m$, we denote, 
following \cite{bambusi2006birkhoff}, its modulus 
\begin{align} \label{moddef}
\lfloor f \rceil(\psi,\bar\psi) &:= \sum_{|j|=r} |f_j| \, z^j, 
\end{align}
where $f_j$ is given by 
\begin{align*}
f(\psi) &= \sum_{|j|=r} f_j z^j, \\
z^j := \cdots z_{-l}^{j_{-l}} \cdots z_{-1}^{j_{-1}} z_{1}^{j_{1}} \cdots z_{l}^{j_{l}} \cdots, \; &\; z_l = \la \psi, e^{il\cdot} \ra, \; \; z_{-l} = \la \bar\psi, e^{-il\cdot} \ra.
\end{align*}
Furthermore, given a multivector 
\begin{align*}
\phi &:= (\phi^{(1)},\ldots,\phi^{(r)}) = (\psi^{(1)},\bar\psi^{(1)}\ldots,\psi^{(r)},\bar\psi^{(r)})
\end{align*}
we introduce the following norm 
\begin{align} \label{norms1}
\|\phi\|_{s,1} &:= \frac{1}{r} \sum_{l=1}^r \|\phi^{(1)}\|_1 \ldots \|\phi^{(l-1)}\|_1
 \|\phi^{(l)}\|_s \|\phi^{(l+1)}\|_1 \ldots \|\phi^{(r)}\|_1.
\end{align}

\begin{definition} \label{tamemapdef}
Let $X: H^s \oplus H^s \to H^s \oplus H^s$ be a homogeneous polynomial of degree 
$r$,
\begin{align*}
X(\psi,\bar\psi) &= \sum_{l \in \Z \setminus \{0\}} X_l(\psi,\bar\psi) e^{il \cdot}.
\end{align*}
Consider the $r$-linear symmetric form $\tilde X_l$ such that 
$\tilde X_l(\psi,\bar\psi,\ldots,\psi,\bar\psi) = X_l(\psi,\bar\psi)$, and set 
\begin{align*}
\tilde X:= \sum_{l \in \Z \setminus \{0\}} \tilde{X_l}(\psi,\bar\psi) e^{il \cdot},
\end{align*}
so that $\tilde X(\psi,\bar\psi,\ldots,\psi,\bar\psi) = X_l(\psi,\bar\psi)$. \\
Let $s \geq 1$, then we say that X is an $s$-tame map  if there exists $K_s>0$ 
such that
\begin{align} \label{stamemap}
\|\tilde X(\phi^{(1)},\ldots,\phi^{(r)})\|_s &\leq K_s \sum_{l=1}^r \|\phi^{(1)}\|_1 \ldots \|\phi^{(l-1)}\|_1 \|\phi^{(l)}\|_s \|\phi^{(l+1)}\|_1 \ldots \|\phi^{(r)}\|_1, \\
&\forall \; \phi^{(1)},\ldots,\phi^{(r)} \in H^s \oplus H^s. \nonumber
\end{align}
If a map is $s$-tame for any $s \geq 1$, then it will be said to be tame.
\end{definition}

\begin{definition} \label{tamevfdef}
Let us consider a vector field $X: H^s \oplus H^s \to H^s \oplus H^s$, and 
denote by $X_l$ its $l$-th component. We define its modulus by 
\begin{align*}
\lfloor X \rceil(\psi,\bar\psi) &:= \sum_{l \in \Z \setminus \{0\}} \lfloor X_l \rceil(\psi,\bar\psi) e^{il\cdot}.
\end{align*}
A polynomial vector field $X$ is said to hace $s$-tame modulus if its modulus 
$\lfloor X \rceil$ is an $s$-tame map. The set of polynomial functions $f$, 
whose Hamiltonian vector fields has $s$-tame modulus will be denoted by $T^s_M$.
 If $f \in T^s_M$ for any $s>1$, we will write $f \in T_m$, and say that $f$ has
 tame modulus.
\end{definition}

\begin{remark}
The property of having tame modulus depends on the coordinate system.
\end{remark}

\begin{definition} \label{tamesnormdef}
Let $X$ be an $s$-tame vector field homogeneous polynomial of degree $r$. 
The infimum of the constants $K_s$ such that the inequality 
\begin{align*}
\|\tilde X(\phi^{(1)},\ldots \phi^{(r)})\| &\leq K_s \|(\phi^{(1)},\ldots,\phi^{(r)})\|_{s,1} \\
&\forall \; \phi^{(1)},\ldots,\phi^{(r)} \in H^s \oplus H^s
\end{align*}
holds will be called tame $s$ norm of $X$, and will be denoted by $|X|^T_s$.
\end{definition}

The tame $s$ norm of a polynomial Hamiltonian f of degree $r+1$ is given by 
\begin{align} \label{tamesnorm}
|f|_s &:= \sup \frac{ \left\| \tilde X_{ \lfloor f \rceil }(\phi) \right\|_s }{ \|\phi\|_{s,1} },
\end{align}
where the $\sup$ is taken over all multivectors 
$\phi=(\phi^{(1)},\ldots,\phi^{(r)})$ such that $\phi^{(j)} \neq 0$ for any $j$. 

\begin{definition}
Let $f \in T^s_M$  be a non-homogeneous polynomial, 
and consider its Taylor expansion 
\begin{align*}
f &= \sum_m f_m,
\end{align*}
where $f_m$ is homogeneous of degree $m$. Let $R>0$, then we denote
\begin{align} \label{tamenorm}
\la|f|\ra_{s,R} &:= \sum_{m \geq 2} |f_r|_s \, R^{m-1}.
\end{align}
Such a definition extends naturally to analytic functions such that 
\eqref{tamenorm} is finite. The set of functions of class $T^s_M$ for which 
\eqref{tamenorm} is finite will be denoted by $T_{s,R}$.
\end{definition}

With the above definitions, 
\begin{align*}
\sup_{B_s(R)} \|X_f(\psi,\bar\psi)\|_s &\leq \la|f|\ra_{s,R}.
\end{align*}
It is easy to check that the set $T_{s,R}$ endowed with the norm 
\eqref{tamenorm} is a Banach space.

Now we introduce the Fourier projection 
\begin{align*}
\Pi_N\psi(x):= \int_{|k| \leq N} \hat\psi(k) e^{ik \cdot x} dk,
\end{align*}
and we split the variables $(\psi,\bar\psi)$ into 
\begin{align*}
(\psi_l,\bar\psi_l) &:= (\Pi_N\psi,\Pi_N\bar\psi), \\
(\psi_h,\bar\psi_h) &:= ((id-\Pi_N)\psi, (id-\Pi_N)\bar\psi).
\end{align*}

The use of Fourier projection is important in view of the following result, 
whose proof can be found in Appendix A of \cite{bambusi2006birkhoff}.

\begin{lemma} \label{tameestlemma}
Fix $N$, and consider the decomposition $\psi=\psi_l+\psi_h$ as above. 
Let $f \in T^s_M$ be a polynomial of degree less or equal than $r+2$. 
Assume that $f$ has a zero of order three in the variables 
$(\psi_h,\bar\psi_h)$, then one has 
\begin{align} \label{tameest}
\sup_{B_s(R)} \|X_f(\psi,\bar\psi)\|_s &\sleq \frac{ \la|f|\ra_{s,R} }{N^{s-1}}.
\end{align}
\end{lemma}

\begin{lemma} \label{poisbrtamelemma}
Let $f,g \in T^s_M$ be homogeneous polynomial of degrees $n+1$ and $m+1$ 
respectively. Then one has $\{f,g\} \in T^s_M$, and 
\begin{align} \label{poisbrtame}
|\{f,g\}|_s &\leq (n+m) |f|_s |g|_s.
\end{align}
\end{lemma}

The proof of this lemma can be found again in Appendix A of \cite{bambusi2006birkhoff}.

\begin{remark}
Given $g$ analytic on $H^s \oplus H^s$, consider the differential equation
\begin{align} \label{orDEs}
\dot \psi &= X_g(\psi,\bar\psi),
\end{align}
where by $X_g$ we denote the vector field of $g$. Now define 
\begin{align*}
\Phi^\ast g(\phi,\bar\phi) &:= g \circ \Phi(\psi,\bar\psi).
\end{align*}
In the new variables $(\phi,\bar\phi)$ defined by 
$(\psi,\bar\psi)=\Phi(\phi,\bar\phi)$ equation \eqref{orDEs}
is equivalent to
\begin{align} \label{pullbDEs}
\dot \phi &= X_{ \Phi^\ast g }(\phi,\bar\phi).
\end{align}

Using the relation
\begin{align*}
\frac{\di}{\di t} (\Phi^t_\chi)^\ast g &= (\Phi^t_\chi)^\ast \{\chi,g\}, 
\end{align*}

 we formally get
\begin{align} \label{lieseriess}
\Phi^\ast g &= \sum_{l=0}^\infty g_l, \\
g_0 &:= g, \\
g_{l} &:= \frac{1}{l} \{\chi,g_{l-1}\}, \; \; l \geq 1.
\end{align}

\end{remark}

In order to estimate the terms appearing in \eqref{lieseriess} 
we exploit the following results

\begin{lemma} \label{liebrestslemma}
Let $h,g \in T_{s,R}$, then for any $d \in (0,R)$ one has that 
$\{h,g\} \in T_{s,R-d}$, and 
\begin{align} \label{liebrests}
\la|\{h,g\}|\ra_{s,R-d} &\leq \frac{1}{d} \la|h|\ra_{s,R} \la|g|\ra_{s,R}
\end{align}
\end{lemma}

\begin{proof}
Write $h = \sum_j h_j$ and $g = \sum_k g_k$, with $h_j$ homogeneous of degree $j$
 and similarly for $g$. Then we have 
\begin{align*}
\{h,g\} &= \sum_{j,k} \{h_j,g_k\},
\end{align*}
where $\{h_j,g_k\}$ has degree $j+k-2$. 
Therefore by \eqref{poisbrtame} in Lemma \ref{poisbrtamelemma}
\begin{align*}
\la|\{h_j,g_k\}|\ra_{s,R-d} &= |\{h_j,g_k\}|_s (R-d)^{j+k-3} \\
&\leq |h_j|_s |g_k|_s (j+k-2) (R-d)^{j+k-3} \\
&\leq  |h_j|_s |g_k|_s \, \frac{1}{d} R^{j+k-2} = \frac{1}{d} \la|h_j|\ra_{s,R} \la|g_k|\ra_{s,R},
\end{align*}
where we exploited the inequality $k(R-d)^{k-1} < R^k/d$, which holds for any 
positive $R$ and $d \in (0,R)$.
\end{proof}

\begin{lemma}
Let $g, \chi \in T_{s,R}$ be analityc functions; denote by $g_l$ the functions 
defined recursively by \eqref{lieseriess}; then for any $d \in (0,R)$ one has 
that $g_l \in T_{s,R-d}$, and 
\begin{align} \label{lieserests}
\la|g_l|\ra_{s,R-d} &\leq \la|g|\ra_{s,R} \left(\frac{e}{d} \la|\chi|\ra_{s,R} \right)^l.
\end{align}
\end{lemma}

\begin{proof}
Fix $l$, and denote $\delta:=d/l$. We look for a sequence $C^{(l)}_m$ such that
\begin{align*}
\la|g_m|\ra_{s,R-m\delta} &\sleq C^{(l)}_m, \; \; 
\forall m \leq l.
\end{align*}
By \eqref{liebrests} in Lemma \ref{liebrestslemma}  we can define the sequence
\begin{align*}
C^{(l)}_0 &:= \la|g|\ra_{s,R}, \\
C^{(l)}_m &= \frac{2}{\delta m} C^{(l)}_{m-1} \, \la|\chi|\ra_{s,R} \\
&= \frac{2l}{dm} \, C^{(l)}_{m-1} \, \la|\chi|\ra_{s,R}.
\end{align*}
One has
\begin{align*}
C^{(l)}_l &= \frac{1}{l!} \left( \frac{2l}{d} \la|\chi|\ra_{s,R} \right)^l  \, \la|g|\ra_{s,R},
\end{align*}
and using the inequality $l^l < l!e^l$ one can conclude.
\end{proof}

\begin{lemma} \label{homeqnrlemma}
Let $f \in T^s_M$ be a polynomial which is at most quadratic in the variables 
$(\psi_h,\bar\psi_h)$. \\
Then $\forall \gamma>0$ there exists a set 
$\cR_\gamma \subset \; [1,+\infty) \times \cV$ satisfying 
\begin{align*}
|\cR_{\gamma} \cap ([n,n+1] \times \cV)| &= \cO(\gamma) \; \; \forall n \in \N_0,
\end{align*}
and $\exists \tau>1$ such that 
$\forall (c,(v_j)_j) \in ([1,+\infty) \times \cV) \setminus \cR_{\gamma}$ 
and $\forall$ $N \geq 1$ the following holds: 
there exist $\chi,Z \in T_{s,R}$ such that
\begin{align} \label{homeqs}
\{ H_0,\chi \} + Z &= f,
\end{align}
and such that $Z$ depends only on the actions and satisfies \eqref{NFtorusprop}. Moreover, $\chi$ and $Z$ satisfy the following estimates 
\begin{align}
\la|\chi|\ra_{s,R} &\leq \frac{N^\tau}{\gamma} \la|f|\ra_{s,R}, \label{vfhomest1} \\
\la|Z|\ra_{s,R} &\leq \la|f|\ra_{s,R}. \label{vfhomest2}
\end{align}
\end{lemma}

\begin{proof}
Expanding $f$ in Taylor series, 
namely $f(\psi,\bar\psi) = \sum_{j,l} f_{j,l} \psi^j \bar\psi^l$, and similarly 
for $\chi$ and $Z$, equation \eqref{homeqs} becomes an equation for the 
coefficients of $f$, $\chi$ and $Z$,
\begin{align*}
i\omega \cdot (j-l) \chi_{j,l} + Z_{j,l} &= f_{j,l}.
\end{align*}
Then we define
\begin{align}
Z_{j} &:= Z_{j,j} = f_{j,j}, \\
\chi_{j,l} &:= \frac{f_{j,l}}{i\omega \cdot(j-l)}, \; \; \text{when} \; j \neq l, \; \; |\omega \cdot (j-l)| \geq \frac{\gamma}{N^\tau }. 
\end{align}
By construction and by Theorem \ref{cond2mel} we get estimates 
\eqref{vfhomest1} and \eqref{vfhomest2}. 
Furthermore, since $f$ is at most quadratic in $(\psi_h,\bar\psi_h)$, we obtain 
that $\sum_{k>N}j_k \leq 2$, and thus $Z$ satisfies \eqref{NFtorusprop}.
\end{proof}

\begin{remark}
Let $s > s^\ast$, and assume that $\chi$, $F$ are analytic on $B_{s}(R)$. 
Fix $d \in (0,R)$, and assume also that
\begin{align*}
\sup_{B_{s}(R)} \|X_\chi(\psi,\bar\psi)\|_{s} \leq d/3,
\end{align*}
Then by Lemma \ref{cauchy} for $|t| \leq 1$
\begin{align}
\sup_{B_{s}(R-d)} \|X_{ (\Phi^t_\chi)^\ast F - F }(\psi,\bar\psi)\|_{s} &= \sup_{B_{s}(R-d)} \|X_{ F \circ \Phi^t_\chi - F }(\psi,\bar\psi)\|_{s} \\
&\stackrel{\eqref{liebrests}}{\leq} 
\frac{5}{d} \, \sup_{B_{s}(R)} \|X_\chi(\psi,\bar\psi)\|_{s} \,
\sup_{B_{s}(R)} \|X_F(\psi,\bar\psi)\|_{s} \\
&< 2 \sup_{B_{s}(R)} \|X_F(\psi,\bar\psi)\|_{s}. \label{vfests}
\end{align}
\end{remark}

\begin{lemma}
Let $\chi \in T_{s,R}$ be the solution of the equation \eqref{homeqs}, with 
$f \in T^s_M$. Denote by $H_{0,l}$ the functions defined recursively via 
\eqref{lieseriess} from $H_0$. Then for any $d \in (0,R)$ one has that 
$H_{0,l} \in T_{s,R-d}$, and
\begin{align}
\la|H_{0,l}|\ra_{s,R-d} &\leq 2 \la|f|\ra_{s,R-d} \left( \frac{e}{d} \la|\chi|\ra_{s,R} \right)^l.
\end{align}
\end{lemma}

\begin{proof}
Using \eqref{homeqs} one gets $H_{0,1} = Z - f \in T^s_M$. Then, arguing as for 
\eqref{lieserests}, one can conclude.
\end{proof}

The main step of the proof of Theorem \ref{BNFtorusthm} is the following 
result, that allows to increase by one the order of the perturbation. 
As a preliminary step, we take the Taylor series of $N(\psi,\bar\psi)$ 
up to order $r+2$,

\begin{align}
N(\psi,\bar\psi) &= \sum_{l=1}^{r} \hat N_{l}(x,\psi,\bar\psi) \label{Taylorexp} \\
&+ N(\psi,\bar\psi) - \sum_{l=1}^{r} \hat N_{l}(x,\psi,\bar\psi) \label{Taylorrem1} \\
&=: N^{(1)}(\psi,\bar\psi) + N^{(1,r)}(\psi,\bar\psi),
\end{align}

where $N_{l}$ is a homogeneous polynomial in $\psi$ and $\bar\psi$
of degree $l+2$ with variable $C^\infty$-coefficients (since $V \in C^\infty$).

Now we consider the analytic Hamiltonian 
\begin{align}
H^{(0)} &:= H_0 + N^{(1)}. \label{ord0}
\end{align}
Then for $R$ sufficiently small one has that 

\begin{align}
\la|N^{(1)}|\ra_{s,R} &\sleq R^2 \label{ord1est1}, \\
\la|N^{(1,r)}|\ra_{s,R} &\sleq R^{r+2} \label{ord1est2}.
\end{align}

\begin{lemma} \label{itlemmator}
Consider the Hamiltonian \eqref{ord0}, and fix $s>s^\ast$. \\
Then $\forall \gamma>0$ there exists a set 
$\cR_\gamma \subset \; [1,+\infty) \times \cV$ satisfying 
\begin{align*}
|\cR_{\gamma} \cap ([n,n+1] \times \cV)| &= \cO(\gamma) \; \; \forall n \in \N_0,
\end{align*}
and $\exists \tau>1$ such that 
$\forall (c,(v_j)_j) \in ([1,+\infty) \times \cV) \setminus \cR_{\gamma}$ 
and $\forall$ $N \geq 1$ the following holds: 
for any $m \leq r$ there exists $R^\ast_m \ll 1$ and, for any $N>1$ 
there exists an analytic canonical transformation
\begin{align*}
\cT^{(m)}:B_s\left( \frac{(2r-m)}{2N^\tau r} R^{\ast,2}_m \right) \to H^s
\end{align*}
such that
\begin{align} \label{cTm}
H^{(m)} &:= H^{(0)} \circ \cT^{(m)} = H^{(0)} + Z^{(m)} + f^{(m)} + \cR^{(m)}_N + \cR^{(m)}_T,
\end{align}
where for any $R < R^\ast_m / N^{\tau}$ the following properties are fulfilled
\begin{enumerate}

\item the transformation $\cT^{(m)}$ satisfies 
\begin{align}
\sup_{B_s(R)} \|\cT^{(m)}-id\|_s &\sleq N^\tau \; R^2; \label{cTtor}
\end{align}
\item $Z^{(m)}$ is a polynomial of degree (at most) $m+2$ that 
depends only on the actions $I = \psi \, \bar\psi$ 
and satisfies \eqref{NFtorusprop}; 
$f^{(m)}$ is a polynomial of degree (at most) $r+2$. Moreover
\begin{align}
\sup_{B_s( (1-m/(2r)) R )} \|X_{Z^{(m)}}(\psi,\bar\psi)\|_s &\sleq R^2, \; \; \forall m \geq 1, \label{itlemmaest1} \\
\sup_{B_s( (1-m/(2r)) R )} \|X_{f^{(m)}}(\psi,\bar\psi)\|_s &\sleq R^{m+2} N^{\tau m}, \; \; \forall m \geq 1; \label{itlemmaest2}
\end{align}
\item the remainder terms $\cR^{(m)}_N$ and $\cR^{(m)}_T$ satisfy 
\begin{align}
\sup_{B_s( (1-m/(2r)) R )} \|X_{\cR^{(m)}_T}(\psi,\bar\psi)\|_s &\sleq R^{r+2} N^{\tau (r+2)}, \label{itlemmarem1} \\
\sup_{B_s( (1-m/(2r)) R )} \|X_{\cR^{(m)}_N}(\psi,\bar\psi)\|_s &\sleq \frac{R^2}{N^{s-1}} \label{itlemmarem2}
\end{align}
\end{enumerate}

\end{lemma}

\begin{proof}
We argue by induction. The theorem is trivial for the case $m=0$, 
by setting $\cT^{(0)}=id$, $Z^{(0)}=0$, $f^{(0)}=N^{(1)}$, $\cR^{(m)}_N=\cR^{(m)}_T=0$. 

Then we split $f^{(m)}$ into two parts, an effective one and a remainder. 
Indeed, we perform a Taylor expansion of $f^{(m)}$ only in the variables 
$(\psi_h,\bar\psi_h)$, namely we write
\begin{align*}
f^{(m)} &= f^{(m)}_0+ f^{(m)}_N,
\end{align*}
where $f^{(m)}_0$ is the truncation of such a series at second order, 
and $f^{(m)}_0$ is the remainder. Since both $f^{(m)}_0$ and $f^{(m)}_N$ are 
truncations of $f^{(m)}$, one has that
\begin{align*}
\la|f^{(m)}_0|\ra_{s,(1-m/(2r)) R} &\sleq \la|f^{(m)}|\ra_{s,(1-m/(2r)) R} \\
\la|f^{(m)}_N|\ra_{s,(1-m/(2r)) R} &\sleq \la|f^{(m)}|\ra_{s,(1-m/(2r)) R}.
\end{align*}
Now consider the truncated Hamiltonian $H_0 + Z^{(m)} + f^{(m)}_0$: 
we look for a Lie transform $\cT_m$ that eliminates the non-normalized 
part of order $m+4$ of the truncated Hamiltonian. 
Let $\chi_m$ be the analytic Hamiltonian generating $\cT_m$. 
Using \eqref{lieseriess} we have 

\begin{align}
& (H_0 + Z^{(m)} + f^{(m)}_0) \circ \cT_m = H_0 + Z^{(m)} \nonumber \\
&+ f^{(m)}_0 + \{ \chi_m,H_0 \} \label{nonnorm1s} \\
&\; \; \; + \sum_{l \geq 1} Z_l^{(m)} + \sum_{l \geq 1} f_{0,l}^{(m)} + \sum_{l \geq 2} H_{0,l}, \label{nonnorm2s} 
\end{align}

with $Z^{(m)}_l$ the $l$-th term in the expansion of the Lie transform of 
$Z^{(m)}$, and similarly for the other quantities. 
It is easy to see that the terms in the first line are already normalized, 
that the term in the \eqref{nonnorm1s} is the non-normalized part of order 
$m+3$ that will vanish through the choice of a suitable $\chi_m$,
 and that the last lines contains all the terms having a zero of order $m+4$ 
at the origin. 

Now we want to determine $\chi_m$ in order to solve the so-called 
``homological equation''
\begin{align*}
\{ \chi_m,H_0 \} + f^{(m)}_0 \; &= \; Z_{m+1},
\end{align*}
with $Z_{m+1}$ depending only on the actions and satisfying \eqref{NFtorusprop}. 
The existence of $\chi_m$ and $Z_{m+1}$ is ensured by Lemma \ref{homeqnrlemma}, 
and by applying \eqref{vfhomest2} and \eqref{itlemmaest1} we get 

\begin{align} \label{esthomeqm}
\la|\chi_m|\ra_{s,(1-m/(2r)) R} &\leq N^\tau \; R^2 \left( N^\tau \; R \right)^m, \\
\la|Z_{m+1}|\ra_{s,(1-m/(2r)) R} &\leq R^2 \; \left( N^\tau \; R \right)^m.
\end{align}
In particular, in view of \eqref{vfests}, we can deduce \eqref{cTtor} at level 
$m+1$. Now define $Z^{(m+1)} := Z^{(m)}+Z_{m+1}$, and $f_C^{(m+1)} := \eqref{nonnorm2s}$.
By \ref{esthomeqm}, recalling that $R < R_m^\ast / N^\tau$, 
we can deduce \eqref{itlemmaest1} at level $m+1$. Moreover, provided that 
$R_m^\ast < 2^{-(m+1)/2}$, one has
\begin{align*}
\delta:= e \, \frac{2 r}{R} \, \la|\chi_m|\ra_{s,(1-m/(2r)) R} &\leq \left( N^\tau \; R \right)^{m+1} < \frac{1}{2}.
\end{align*}
By \eqref{lieserests} and \eqref{itlemmaest1} one thus gets 
\begin{align*}
\la|f_C^{(m+1)}|\ra_{s,(1-(m+1)/(2r)) R} &\sleq \sum_{l\geq1} R^2 \delta^l + \sum_{l\geq1} R^2 \delta^l \left( N^\tau \; R \right)^m + \sum_{l\geq2} R^2 \delta^{l-1} \left( N^\tau \; R \right)^m \\
&\sleq R^2 \left( N^\tau \; R \right)^{m+1}.
\end{align*}
Write now $f_C^{(m+1)} = f^{(m+1)} + \cR_{m,T}$, where $f^{(m+1)}$ is the Taylor 
polynomial of order $r+2$ of $f_C^{(m+1)}$, and where $\cR_{m,T}$ has a zero 
of order $r+3$ at the origin. Clearly $f^{(m+1)}$ satisfies 
\eqref{itlemmaest2} at level $m+1$, since it is a truncation of $f_C^{(m+1)}$. 
The remainder may be bounded by using Lagrange and Cauchy estimates,
\begin{align*}
\sup_{B_s( (1-m/(2r)) R )} \|X_{ \cR_{m,T} }(\psi,\bar\psi)\|_s &\sleq 
\frac{1}{(r+2)!} \, R^{r+2} \, \sup_{B_s \left( R_m^\ast /(2 N^\tau) \right)} \|\d^{r+2} X_{ f_C^{(m+1)} }(\psi,\bar\psi)\|_s \\
&\sleq R^{r+2} \, \left( \frac{2 N^\tau}{R_m^\ast} \right)^{r+2} \, \sup_{B_s \left( R_m^\ast/ N^\tau \right)} \|X_{ f_C^{(m+1)} }(\psi,\bar\psi)\|_s  \\
&\sleq \left( N^\tau \; R \right)^{r+2}.
\end{align*}
Now define $\cR_T^{(m+1)} := \cR_T^{(m)} \circ \cT_m + \cR_{m,T}$. 
By \eqref{vfests} we can deduce \eqref{itlemmarem1} at level $m+1$.
Then set $\cR_N^{(m+1)} := (\cR_N^{(m)} + f_N^{(m)}) \circ \cT_m$. 
By \eqref{itlemmaest2} and \eqref{itlemmarem2}, together with \eqref{vfests} 
and \eqref{tameest} in Lemma \ref{tameestlemma}, 
we obtain \eqref{itlemmarem2} at level $m+1$. 
\end{proof}

Now we conclude the proof of Theorem \ref{BNFtorusthm}. \\
By taking the canonical transformation $\cT^{(r)}$ defined in the iterative 
Lemma \ref{itlemmator} we have that 
\begin{align} \label{}
H^{(r)} &= H_0 + Z^{(r)} + \cR_N^{(r)} + \cR_T^{(r)}  + N^{(1,r)} \circ \cT^{(r)},
\end{align}
with $Z^{(r)}$ depending only on the actions and satisfying \eqref{NFtorusprop}, 
and for any $R < R^\ast_m / N^\tau$ the following holds
\begin{align*}
\sup_{B_s(R)} \|\cT^{(r)}(\psi,\bar\psi)-(\psi,\bar\psi)\|_{s} &\sleq N^{2\tau} \; R^3, \\
\sup_{B_s(R)} \|X_{ \cR_N^{(r)} }(\psi,\bar\psi)\|_{s} &\sleq \frac{R^2}{N^{s-1}}, \\
\sup_{B_s(R)} \|X_{ \cR_T^{(r)} }(\psi,\bar\psi)\|_{s} &\sleq \left(N^\tau \; R \right)^{r+2}, \\
\sup_{B_s(R)} \|X_{ N^{(1,r)} \circ \cT^{(r)} }(\psi,\bar\psi)\|_{s} &\sleq \left(N^\tau \; R \right)^{r+2}.
\end{align*}
To conclude we have just to choose $N$ and $s$ sich that $\cR_N^{(r)}$ and 
$\cR_T^{(r)}$ are of the same order of magnitude. First take $N=R^{-a}$, with 
$a$ still to be determined; then, in order to obtain that $\cR_T^{(r)}$ is 
of order $\cO(R^{r+3/2})$ we choose $a:= \frac{1}{2\tau (r+2)}$. 
By taking $s > 2\tau r(r+2)+1$ we get that also $N^{(1,r)}$ is of the same 
order of magnitude. 

Now take $K^\ast=1/24$, and construct the canonical transformation 
$(\psi,\bar\psi)=\cT^{(r)}(\psi',\bar\psi')$. Denote by $I'$ the actions 
expressed in the variable $(\psi',\bar\psi')$, and define the function 
$\cN(\psi',\bar\psi') := \|I'\|^2_s$. By \eqref{CTtorusest} one has 
that $\cN(\psi'_0,\bar\psi'_0) \leq \frac{32}{31} R^2$, provided that 
$R$ is sufficiently small. Since 
\begin{align*}
\frac{\d \cN}{\d t}(\psi',\bar\psi') = \{R^{(r)},\cN\}(\psi',\bar\psi'),
\end{align*}
and therefore, as far as $\cN(\psi',\bar\psi') < \frac{64}{9} R^2$,
\begin{align} \label{esctimeest}
\left| \frac{\d \cN}{\d t}(\psi',\bar\psi') \right| &\leq K_s' \; R^{r+5/2}.
\end{align}
Denote by $T_f$ the escape time of $(\psi',\bar\psi')$ from 
$B_s(R/3)$; observe that for all times smaller than 
$T_f$, \eqref{esctimeest} holds. So one has
\begin{align*}
\frac{64}{9} R^2 = \cN(\psi'(T_f),\bar\psi'(T_f)) &\leq \cN(\psi'_0,\bar\psi'_0) + K_s' \; R^{r+5/2} T_f, 
\end{align*}
which shows that $T_f$ should be of order (at least) $R^{r+1/2}$. 
Going back to the original variables one gets \eqref{smallsoltorus}. 
To show \eqref{acttorus}, one has to recall that 
\begin{align*}
|I(t)-I(0)| &\leq |I(t)-I'(t)| + |I'(t)-I'(0)| + |I'(0)-I(0)|,
\end{align*}
and that by \eqref{CTtorusest} and \eqref{smallsoltorus} one can estimate the 
first and the third term; the second term can be bounded by computing the time 
derivative of $\|I'\|^2_s$ with the Hamiltonian, and observing that it is of 
order $\cO(R^{r+5/2})$. 

Now, consider the initial actions $(I_0,\bar I_0) := (I(0),\bar I(0))$. 
By passing to the Fourier transform, 
\begin{align*}
I_j(t) &:= \widehat{I(t)}(j), \; \; j \geq 1, 
\end{align*}
we have that for any $r_1 \leq r$
\begin{align} \label{estactr1}
| (I_j(t),\bar I_j(t)) - (I_j(0),\bar I_j(0)) | &\sleq \frac{R^{2r_1}}{j^{2s}}, \; \; |t| \sleq R^{-(r-r_1+1/2)}.
\end{align}
If we define the torus 
\begin{align*}
I_c &:= \{ (\psi,\bar\psi) \in H^s : (I_j(\psi,\bar\psi),\bar I_j(\psi,\bar\psi) = (I_j(0),\bar I_j(0) ), \; \; \text{for any} \; \; j \geq 1 \},
\end{align*}
we get 
\begin{align*}
d_{s_1}( (\psi(t),\bar\psi(t)), I_c )  &\leq \left[ \sum_j j^{2s_1} \left( |\sqrt{I_j(t)} - \sqrt{I_j(0)} |^2 + |\sqrt{\bar I_j(t)} - \sqrt{\bar I_j(0)}|^2 \right) \right]^{1/2},
\end{align*}
and by using \eqref{estactr1} we obtain
\begin{align*}
d_{s_1}( (\psi(t),\bar\psi(t)), I_c ) ^2 &\leq \left( \sup_j j^{2s} |I_j(t) - I_j(0)|^2 + j^{2s} |\bar I_j(t) - \bar I_j(0)|^2 \right) \; \sum_j \frac{1}{j^{2(s-s_1)}},
\end{align*}
which is convergent for $s_1 < s-1/2$. \\

\bibliography{P_NLKG_2017b}
\bibliographystyle{alpha}

\end{document}